\documentclass[preprint,12pt]{elsarticle}

\usepackage{amssymb}
\usepackage[cmex10]{amsmath}
\usepackage{amsmath,amsfonts,amssymb,amsthm,dsfont,stmaryrd}
\usepackage{hyperref}
\usepackage{color}

\newtheorem{theorem}{Theorem}[section]

\newtheorem{lemma}{Lemma}[section]
\newtheorem{corollary}{Corollary}[section]

\newtheorem{definition}{Definition}[section]

\newtheorem{example}{Example}[section]

\def\s{s}

\def\R{\mathbb{R}}

\def\M{\mathcal{M}}

\def\m{\mathrm{m}}

\def\A{\mathcal{A}}

\def\F{\mathcal{F}}

\def\diag{\mathrm{diag}}

\def\w{\omega}

\def\x{\mathrm{x}}
\def\sd{\mathrm{sd}}
\def\ICA{ICA$_{SG}$}
\def\D{\mathcal{D}}
\newcommand{\argmin}{\operatornamewithlimits{argmin}}
\newcommand{\argmax}{\operatornamewithlimits{argmax}}
\def\z{\mathrm{z}}

\def\1{\mathds{1}}

\def\for{\mbox{  for }}

\def\cov{\mathrm{cov}}
\def\mean{\mathrm{mean}}

\def\det{\mathrm{det}}

\def\cov{\mathrm{cov}}

\numberwithin{equation}{section}

\usepackage{algorithmic}
\usepackage{algorithm}

\usepackage{epstopdf}
\usepackage[tight,footnotesize]{subfigure}

\journal{Pattern Recognition}

\begin{document}

\begin{frontmatter}



\title{ICA based on the data asymmetry}

\author{P. Spurek, J. Tabor}
\address{Faculty of Mathematics and Computer Science, 
Jagiellonian University, 
\L ojasiewicza 6, 
30-348 Cracow, 
Poland}
\ead{przemyslaw.spurek@ii.uj.edu.pl,jacek.tabor@ii.uj.edu.pl}

\author{P. Rola}
\address{
Department of Mathematics of the Cracow University of Economics, 
Rakowicka 27,
31-510 Cracow, 
Poland}
\ead{przemyslaw.rola@outlook.com}
\author{M. Ociepka}
\address{Faculty of Mathematics and Computer Science, 
Jagiellonian University, 
\L ojasiewicza 6, 
30-348 Cracow, 
Poland}



\begin{abstract}
Independent Component Analysis (ICA) - one of the basic tools in data analysis - aims to find a coordinate system in 
which the components of the data are independent. 
Most of existing methods are based on the minimization of the function of fourth-order moment (kurtosis).
Skewness (third-order moment) has received much less attention.

In this paper we present a competitive approach to ICA based on the Split Gaussian distribution, which is well adapted to asymmetric data. 
Consequently, we obtain a method which works better than the classical 
approaches, especially in the case when the underlying density is not symmetric, which is a typical situation in the color distribution in images.
\end{abstract}


\end{frontmatter}



\section{Introduction}

Independent component analysis (ICA) is one of the most popular methods of data analysis and preprocessing. Historically, Herault and Jutten \cite{herault1986space} seem to be the first (around 1983) to have addressed the problem of ICA to separate mixtures of independent signals.

  In signal processing ICA is a computational method for separating a multivariate signal into additive subcomponents and has been applied in magnetic resonance \cite{beckmann2004probabilistic}, MRI \cite{beckmann2005tensorial,rodriguez2012noising}, EEG analysis \cite{brunner2007spatial,delorme2007enhanced,zhang2013bayesian},
fault detection \cite{choi2005fault}, financial time series \cite{kiviluoto1998independent} and  seismic recordings \cite{haghighi2008ica}.
Moreover, it is hard to overestimate the role of ICA in pattern recognition and image analysis; its applications include face recognition \cite{yang2005kernel,dagher2006face}, facial action recognition~\cite{chuang2006recognizing},  image filtering \cite{tsai2006independent}, texture segmentation \cite{jenssen2003independent}, object recognition~\cite{bressan2003using,tao2016ensemble}, image modeling \cite{kim2005iterative}, embedding graphs in pattern-spaces \cite{luo2003spectral,luo2002independent}, multi-label learning \cite{xu2016local} and feature extraction \cite{lai2014multilinear}. The calculation of ICA is discussed in several papers \cite{secchi2016hierarchical, hyvarinen2004independent,lee1999independent,cardoso1989source,pham1997blind,comon1994independent,du2016hyperspectral}, where the problem is given various names, in particular it is also called ``source separation problem''.

ICA is similar in many aspects to principal component analysis (PCA). In PCA we look for an orthonormal change of basis so that the components are not
linearly dependent (uncorrelated).
ICA can be described as a search for the optimal basis (coordinate system) in which the components are independent. Let us now, for the readers convenience, describe how the 
ICA works. Data is represented by the random vector $\x$ 
 and the components as the random vector~$s$. 
 The aim is to transform the observed data $\x$ into maximally independent components $s$ with respect to some measure 
of independence. Typically we use a linear static transformation $W$, called the {\em transformation matrix}, combined with the formula $s = W \x$. 

Most ICA methods are based on the maximization of non-Gaussianity. This follows from the fact that one of the theoretical foundations of ICA is given by the dual view at the Central Limit Theorem \cite{hyvarinen2000independent}, which states that the distribution of the sum (average or linear combination) of $N$ independent random variables approaches Gaussian as  $N\rightarrow \infty$. Obviously if all source variables are Gaussian, ICA will not work. 

\begin{figure}[!h]
\normalsize
\begin{center}
\includegraphics[width=5in]{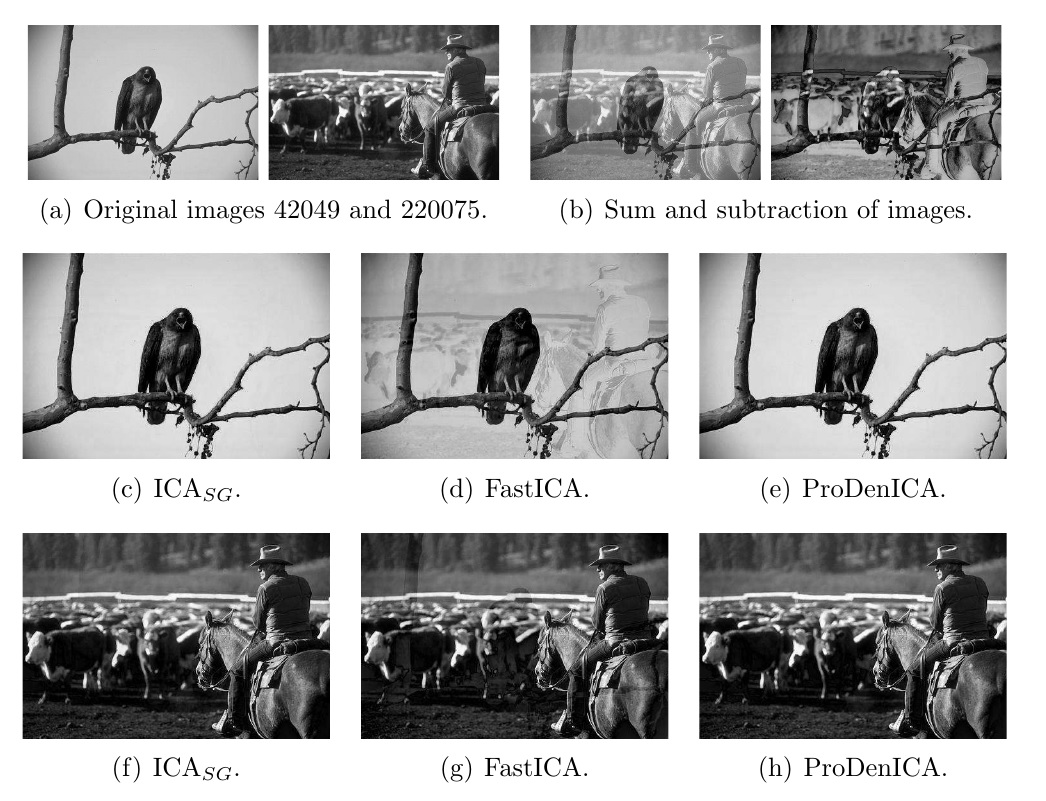}
\end{center}
\caption{Comparison of images separation by our method (\ICA), with FastICA and ProDenICA.}
\label{fig:image_ICA_int}
\end{figure}

The classical measure of non-Gaussianity is kurtosis (the forth central moment), which can be both positive or negative. Random variables that have a negative kurtosis are called subgaussian, and those with the positive one are called supergaussian. Supergaussian random variables have typically a ``spiky'' pdf with heavy tails, i.e. the pdf is relatively large at zero and at large values of the variable, while being small for intermediate values (ex. the Laplace distribution). Typically non-Gaussianity is measured by the absolute value of kurtosis (the square of kurtosis can also be used). 
Thus many methods of finding independent components are based on fitting a density with similar kurtosis as the data, and consequently are very sensitive to the existence of outliers. Moreover, typically data sets are bounded, and therefore the credible estimation of tails is not easy. Another problem with these methods, is that they usually assume that the underlying density is symmetric, which is rarely the case.

In our work we introduce and explore a new approach \ICA, based on the asymmetry of the data, which can be measured by the third central moment (skewness). Any symmetric data, in particular gaussian, has skewness equal to zero.
Negative values of skewness indicate the data skewed to the left and the positive ones indicate the data skewed to the right\footnote{By skewed to the left, we mean that the left tail is long relative to the right tail.
Similarly, skewed to the right means that the long tail is on the right-hand side.
}. Consequently, skewness is a natural measure of non-Gaussianity. In our approach, instead of approximating the data by product of densities
with heavy tails, we approximate it by a product of 
asymmetric densities (so called Split Gaussians).

Contrary to classical approaches which consider third or fourth central moment, our algorithm is based on second moments. This is a consequence of the fact that Split Gaussian distributions arise from merging two opposite halves of normal distributions in their common mode (for more information see Section \ref{SGD}). Therefore  we use only second order moments to describe skewness in dataset, and therefore we obtain an effective ICA method which is resistant to outliers.

\begin{figure*}[!h]
\normalsize
\begin{center}
\includegraphics[width=4.5in]{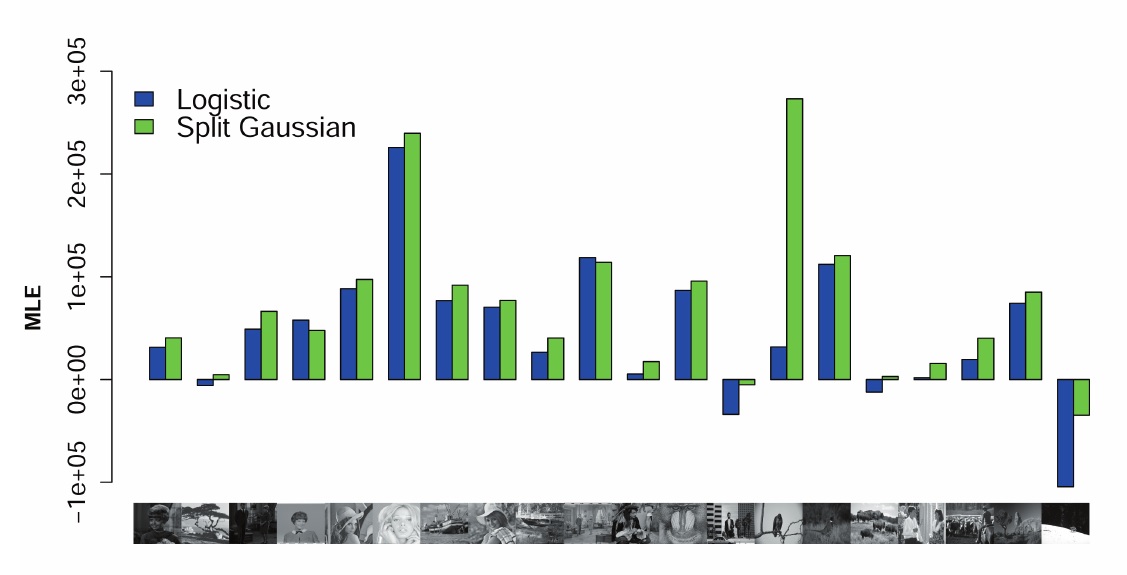} 
\end{center}
\caption{MLE estimation for image histograms with respect to Logistic and Split Gaussian distributions.}
\label{fig:MLE}
\end{figure*}

The results of classical ICA and \ICA \ in the case of image separation (for more detail comparison we refer to Section \ref{ex}) is presented in Fig. \ref{fig:image_ICA_int}. In the experiment we mixed two images (see Fig. \ref{fig:image_ICA_int} a) by adding and subtracting them (see Fig. \ref{fig:image_ICA_int} b). Our approach gives essentially better results than the classical FastICA approach, compare Fig. \ref{fig:image_ICA_int} c) to Fig. \ref{fig:image_ICA_int} d) and  Fig. \ref{fig:image_ICA_int} f) to Fig. \ref{fig:image_ICA_int} g). In the case of classical ICA we can see artifacts in background, which means that the method does not separate signal properly. On the other hand, ProDenICA and \ICA \ almost perfectly recovered images, compare Fig. \ref{fig:image_ICA_int} c) to Fig. \ref{fig:image_ICA_int} e) and Fig. \ref{fig:image_ICA_int} f) to Fig. \ref{fig:image_ICA_int} h).

\begin{figure}[!h] 
\normalsize
\begin{center}
\includegraphics[width=5in]{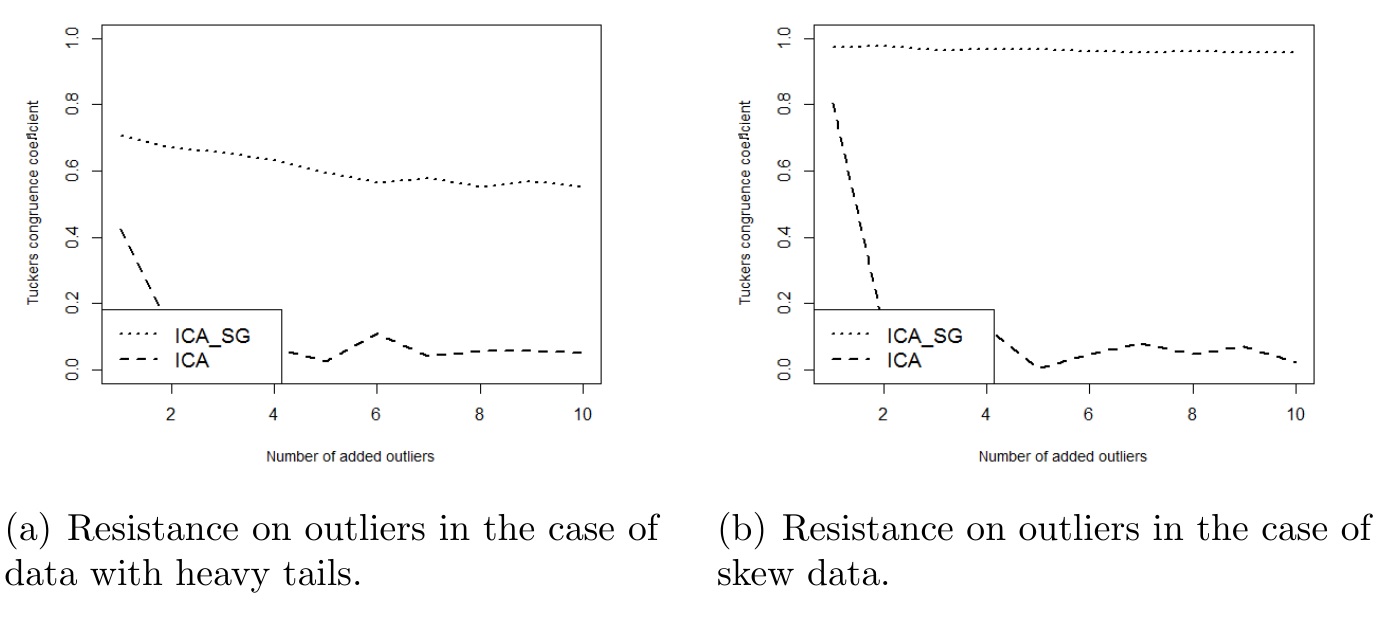} 
\end{center}
\caption{Comparison between our approach and classical ICA in the case of resistance on outliers.}
\label{fig:out}
\end{figure}


In general, \ICA \  in most cases gives better results than other ICA methods, see Section \ref{ex}, while its numerical complexity lies below the methods which obtain
comparable results, that is ProDenICA and PearsonICA.
This is caused in particular by the fact that asymmetry is more common than heavy tails in real data sets -- we performed the symmetry test by using R package {\tt lawstat} \cite{test} with 5 percent confidence ratio, and it occurred that all image datasets we used in our paper have asymmetric densities.
We also verified it in the case of density estimation of our images. We found optimal parameters of Logistic and Split Gaussian distributions and compared the values of MLE function in Fig. \ref{fig:MLE}. As we see, in most cases Split Gaussian distribution fits the data better than the Logistic one.

Summarizing the results obtained in the paper, our method works better than classical approaches for asymmetric data, and is 
more resistant to outliers (see Example \ref{ex:out}).

\begin{example} \label{ex:out}

We consider the data with heavy tails (a sample from the Logistic distribution) and skewed ones (a sample from the Split Normal distribution). We added to the data outliers uniformly generated from 
rectangle $[\min(X_1)-\sd(X_1),\max(X_1)+\sd(X_1)]\times[\min(X_2)-\sd(X_2),\max(X_2)+\sd(X_2)]$, where $\sd(X_i)$ is a standard deviation of the $i$-th coordinate of $X$. In Fig.~\ref{fig:out} we present how the absolute value of the Tucker's congruence coefficient (the similarity measure of extracted factors, see Section \ref{ex}) is changing when we add the outliers. 

As we see, \ICA \ is more stable and deals better with outliers in the data, which
follows from the fact that classical ICA typically depends on the moments of order four, while our approach uses moments of order two.
\end{example}

This paper is arranged as follows. In the second section, we discuss related works. In the third, the theoretical background of our approach to ICA is presented. We introduce a cost function which uses the General Split Gaussian distribution and show that it is enough to minimize it respectively to only two parameters: vector $\m \in \R^d$  and $d \times d$ matrix  $W$. We also calculate the gradient of the cost function, which is necessary for the efficient use in the minimization procedure.
The last section describes numerical experiments. The effects of our algorithm are illustrated on simulated and real datasets.

\section{Related works}\label{RW}

\begin{figure*}[!t]
\begin{center}
\includegraphics[width=5in]{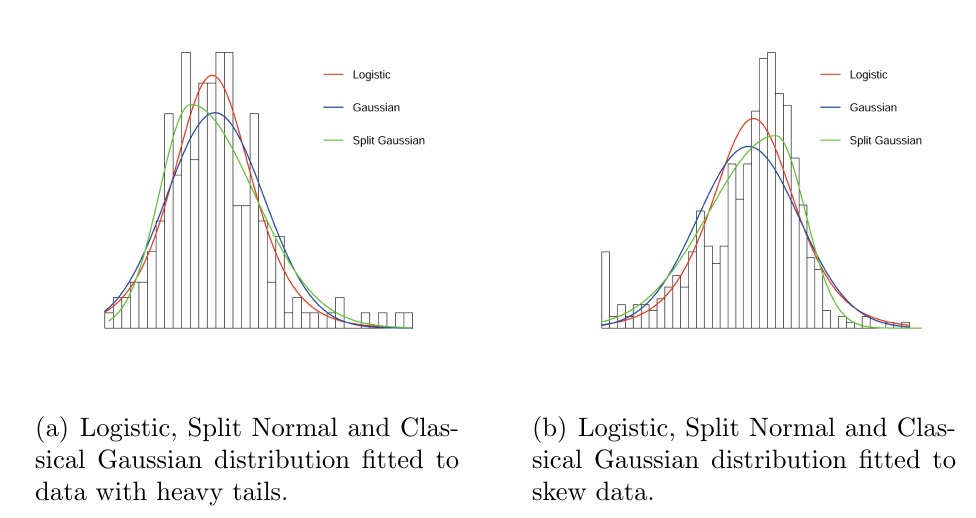}  
\end{center}
\caption{Logistic, Split Normal and Classical Gaussian distribution fitted to data with heavy tails and skew one.}
\label{fig:den_1d}
\end{figure*}

Various ICA methods were discussed in \cite{secchi2016hierarchical, hyvarinen2004independent,lee1999independent,cardoso1989source,pham1997blind,comon1994independent}. Herault and Jutten seem to be the first who introduced the ICA around 1983. They proposed an iterative real-time algorithm based on a neuro-mimetic architecture, which nevertheless, can show the lack of convergence
in a number of cases \cite{jutten1991blind}. It is worth mentioning that in their
framework, higher-order statistics were not introduced
explicitly. Giannakis et al. \cite{giannakis1989cumulant} addressed the issue of
identifiability of ICA in 1987 using third-order cumulants. However, the resulting algorithm required an exhaustive
search. 

Lacoume and Ruiz \cite{lacoume1992separation} sketched
a mathematical approach to the problem using
higher-order statistics, which can be interpreted as a measure of fitting independent components. 
Cardoso \cite{cardoso1991super,cardoso1999high} focused on the algebraic properties of the fourth-order cumulants (kurtosis) what is still a popular approach~\cite{sharma2006subspace}.
Unfortunately kurtosis has some drawbacks in practice, when its value has to be estimated from a measured sample. The main problem is that kurtosis can be very sensitive to the outliers. Its value may depend on only a few observations in the tails of the distribution. 
In high-dimensional problems, where separation process contains PCA (for dimension reduction), whitening (for scale normalization), and standard ICA
this effect is called a small sample size problem \cite{yang2005ica,deng2012small}. This is caused by the fact that for the high-dimensional data sets ICA algorithms tend to extract the independent features simply by the projections that isolate single or very few samples (outliers). To address the difficulty random pursuit and locality pursuit methods were applied
\cite{deng2012small}.

Another commonly used solution is to use skewness \cite{stone2002spatiotemporal,kollo2008multivariate,liu2011investigation,karvanen2004independent} instead of kurtosis. 
Unfortunately, skewness has received much less attention than kurtosis, and consequently methods based on skewness are usually not well theoretically justified. 

One of the most popular ICA method dedicated to the skew data is PearsonICA \cite{karvanen2000pearson,karvanen2002blind}, which minimizes mutual information using a Pearson \cite{stuart1968advanced} system-based parametric model.  The
model covers a wide class of source distributions
including skewed distributions.
The Pearson system is defined by the differential equation
$$
f'(x) = \frac{(a_1x - a_0)f(x)}{b_0 + b_1x + b_2x^2}, 
$$
where $a_0$, $a_1$, $b_0$, $b_1$ and $b_2$ are the parameters of the
distribution.
The parameters of the Pearson system can be estimated
using the method of moments.
Therefore such algorithms have strong limitations connected with the optimization procedure. The main problems are number of parameters which have to be fitted and numerical efficiency of the minimization procedure.

An important measure of fitting independent components is given by negentropy \cite{gaeta1990source}. FastICA \cite{hyvarinen1999fast}, one of the most popular implementations of ICA,  uses this approach.
Negentropy is based on the information-theoretic quantity of (differential) entropy. This concept leads to the mutual information
which is the natural information-theoretic measure of the independence of random variables. Consequently, one can use it as the criterion for finding the ICA transformation \cite{comon1994independent,bell1995information}. 
It can be shown that minimization of the mutual information is roughly equivalent to maximization of negentropy and it is easier to estimate since we do not need additional parameters. ProDenICA \cite{bach2002kernel,hastie2009elements} is based not on a
single nonlinear function, but on an entire function space of candidate nonlinearities. In particular, the method works with the functions in a reproducing kernel Hilbert space, and make use of the ``kernel trick'' to search over this space efficiently. The use of a function space makes it possible to adapt to a variety of sources and thus makes ProDenICA algorithms more robust to varying source distributions.

A somewhat similar approach to ICA is based on the maximum likelihood estimation~\cite{pham1997blind}. It is closely connected to the infomax principle since the likelihood is proportional to the negative of mutual information. 
In recent publications, the maximum likelihood estimation is one of the mot popular \cite{hyvarinen2004independent,harroy1996maximum,comon2010handbook,samworth2012independent,zarzoso2006optimal,murillo2004sinusoidal,cardoso2006maximum} approaches to ICA.  Maximum likelihood approach needs the source pdf. 
In the classical ICA it is common to use the super-Gaussian logistic density or other heavy tails distributions.

In this paper we present \ICA, a method which joins the positive aspects of
classical \ICA \ approaches with recent ones like ProDenICA or Pearson ICA. 
First of all we use a General Split Gaussian distribution, which uses second order moments to describe skewness in dataset, and therefore is relatively robust to noise or outliers. The GSG distribution can be fitted by minimizing a simple function, which depends on only two parameters $\m \in \R^d$, $W \in \M(\R^d)$, see Theorem \ref{the:min}. Moreover we calculate its gradient, and therefore we can use numerically efficient gradient type algorithms, see Theorem~\ref{ther:grad}.

%
\section{Theoretical justification} \label{se:theor}

%
%
%
%
%
%
%
%
%
%
%
%
%
%

\begin{figure*}[!t]
\normalsize
\begin{center}
\includegraphics[width=5in]{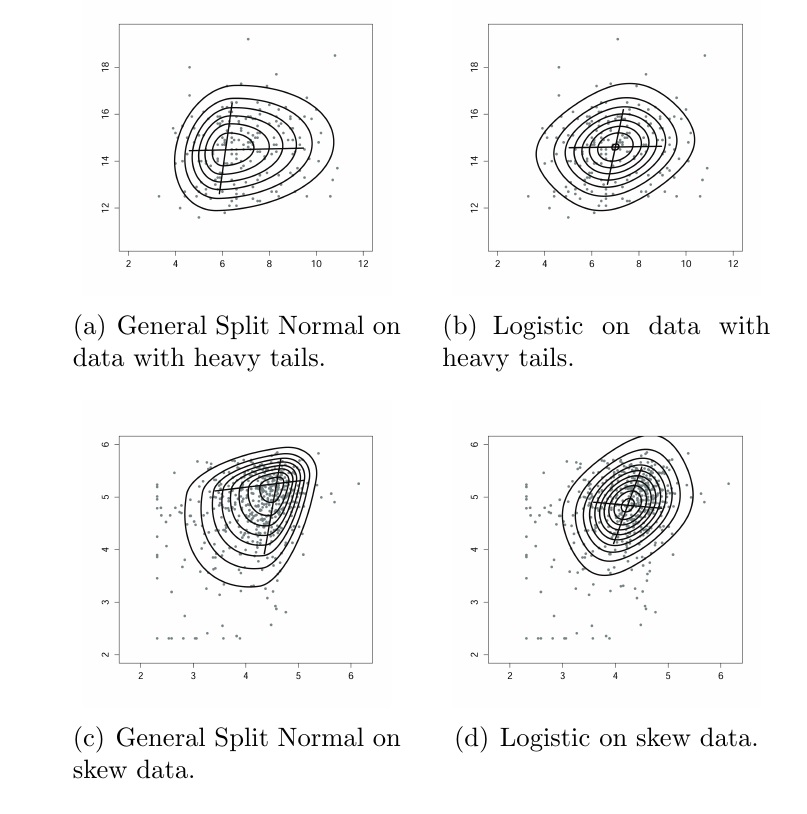}
\end{center}
\caption{Logistic and General Split Normal distributions fitted to data with heavy tails and skew ones.}
\label{fig:den_2d}
\end{figure*}

Let us describe the idea\footnote{In fact it is one of the possible approaches, as there are many explanations which lead to similar formula.} behind ICA \cite{hyvarinen2000independent}. Suppose that we have a random vector $X$
in $\R^d$ which is generated by the model with the density $F$. Then it is well-known that components of $X$ are independent iff there exist one-dimensional densities $f_1,\ldots,f_d \in \D_\R$, where by $\D_\R$ we denote the set of densities on $\R$, such that
$$
F(\x)=f_1(x_1) \cdot \ldots \cdot f_d(x_d), \for
\x=(x_1,\ldots,x_d) \in \R^d.
$$
Now suppose that the components of $X$ are not independent, but that
we know (or suspect) that there is a basis $A$ (we put $W=A^{-1}$) such that in that base the
components of $X$ become independent. This may be formulated in the form
\begin{equation} \label{eq:gen}
F(\x)=\det(W) \cdot f_1(\w_1^T(\x-\m)) \cdot \ldots \cdot f_d(\w_d^T(\x-\m)) \for x \in \R^d,
\end{equation}
where $\w_i^T(\x-\m)$ is the $i$-th coefficient of $\x-\m$ (the basis is centered in $\m$) in the basis $A$ ($\w_i$ denotes the $i$-th column of $W$).
Observe, that for a fixed family of one-dimensional densities $\F \subset \D_\R$, the set of all densities given by \eqref{eq:gen} for $f_i \in \F$, forms an affine invariant set of densities.

Thus, if we want to find such a basis that components become independent, we need to search for a matrix $W$ and one-dimensional densities such that the approximation 
$$
F(\x) \approx \det(W) \cdot f_1(\w_1^T(\x-\m)) \cdot \ldots \cdot f_d(\w_d^T(\x-\m)), \for \x \in \R^d,
$$
is optimal. However, before proceeding to practical implementations, we need to precise:
\begin{enumerate}
\item how to measure the above approximation,
\item how to deal with data $X$, since we do not have the density,
\item how to work with the family of all possible densities.
\end{enumerate}

The answer to the first point is simple and is given by the Kullback-Leibler divergence, which is defined to be the integral:
$$
D_{\mathrm{KL}}(P\|Q) = \int_{-\infty}^\infty p(x) \, \log\frac{p(x)}{q(x)} \, {\rm d}x,
$$
where $p$ and $q$ denote the densities of $P$ and $Q$. This can be written
as
$$
D_{\mathrm{KL}}(P\|Q)=h(P)-MLE(P,Q),
$$
where $h$ is the classical Shannon entropy.
Thus to minimize the Kullback-Leibler divergence, we can equivalently maximize
the MLE. This is helpful, since for a discrete data $X$ we have nice estimator of the LE (likelihood estimation):
$$
LE(X,Q)=\frac{1}{|X|} \sum_{\x \in X} \ln(q(x)). 
$$
Thus we arrive at the following problem.

\medskip

\noindent{\bf Problem [reduced]. }{\em
Let $X$ be a data set. Find an unmixing matrix $W$, center $\m$, and densities $f_1,\ldots,f_d \in \D_\R$ so that the value 
$$
\begin{array}{l}
LE(X,f_1,\ldots,f_d,\m,W)=\\[6pt]
\frac{1}{|X|} \sum \limits_{\x \in X} \ln(f_1(\w_1^T(\x-\m))  \ldots  f_d(\w_d^T(\x-\m)))+\ln(\det(W)) \!= \\[6pt]
\frac{1}{|X|}\sum \limits_{i=1}^d \sum \limits_{\x \in X} \ln(f_i(\w_i^T(\x-\m)))+\ln(\det(W)) 
\end{array}
$$
is maximized.
}

However, there is still a problem with the last point, as the search over the space of all densities $\D_\R$ is not feasible. Thus, we naturally have to reduce our search to a subclass of all densities $\F$ (which should be parametrized by a finite amount of parameters).

\medskip

\noindent{\bf Problem [final]. }{\em
Let $X \subset \R^d$ be a data set and $\F \subset \D_\R$ be a set of densities. Find an unmixing matrix $W$, center $\m$, and densities $f_1,\ldots,f_d \in \F$ so that the value 
$$
\frac{1}{|X|}\sum_{i=1}^d \sum_{\x \in X} \ln(f_i(\w_i^T(\x-\m)))+\ln(\det(W)) 
$$
is maximized.
}

It may seem that the most natural choice is Gaussian densities. However, this is not the case as Gaussian densities are affine invariant, and therefore do not ``prefer'' any fixed choice of coordinates\footnote{In fact one can observe that the choice of gaussian densities leads to PCA, if we restrict to the case of orthonormal bases}. In other words we have to choose a family of densities which is distant from Gaussian ones.

In the classical ICA approach it is common to use the super-Gaussian logistic distribution:
$$
f(x; \mu,s) = \frac{e^{\frac{x-\mu}{s}}} {s\left(1+e^{\frac{x-\mu}{s}}\right)^2} =\frac{1}{4s} \operatorname{sech}^2\!\left(\frac{x-\mu}{2s}\right).
$$
The main difference between the gaussian and super-gaussian is the existence of the heavy tails. This can be also viewed as the difference in the fourth moments.

However, such a choice leads to some negative consequences, namely the model is very sensitive to outliers. Moreover, if the data is not-symmetric, the approximation could not give the expected results, as the model consists only of
symmetric densities.

The idea behind this paper was to choose the model of densities which wouldn't
have the two above disadvantages. So, instead of choosing the family which differs from the Gaussians by the size of tail (fourth moment), we chose a family which would allow estimation of asymmetric densities -- Split Gaussian distribution \cite{gibbons1973estimation}. 

\begin{example} \label{ex:2}
In Fig. \ref{fig:den_1d} and Fig. \ref{fig:den_2d} we present a comparison between the Logistic and the Split Normal  distribution in 1d and 2d respectively. In experiments we use the classical skew dataset Lymphoma \cite{maier2007allelic,pyne2009automated} and the classical heavy tails dataset Australian athletes \cite{clauset2009power}.  In the case of heavy tails both methods work nice, since real dataset represent heavy tails which 
are not symmetric and the skew model is able to detect it. On the other hand, in the case of skew data Split Normal gives essentially better results.       
\end{example}   

%
\section{Split Gaussian distribution}\label{SGD}

\begin{figure*}[!t]
\normalsize
\begin{center}
\includegraphics[width=5in]{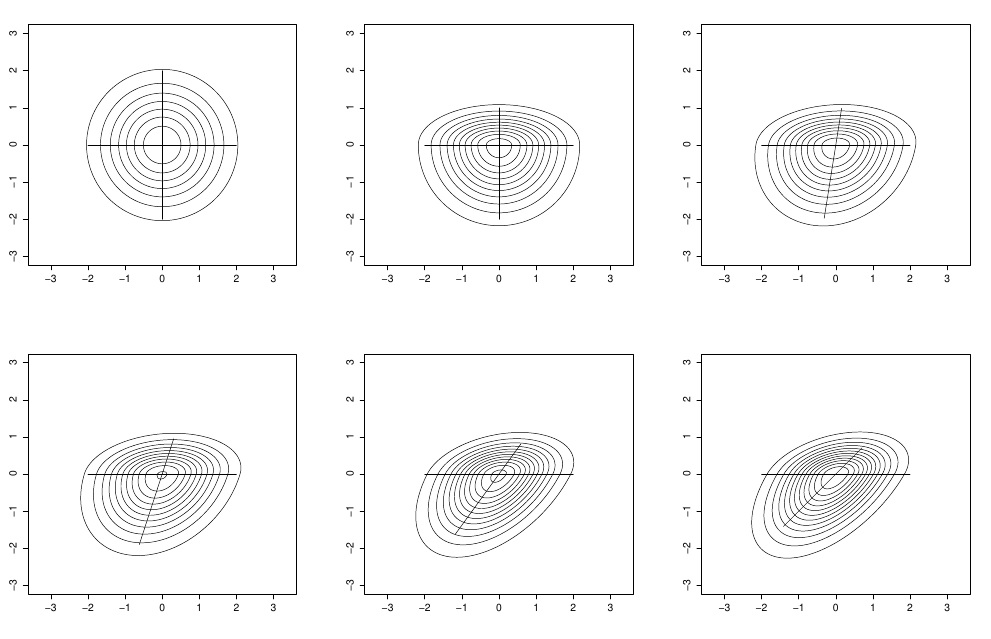}
\end{center}
\caption{Level sets of the General Split Normal distribution with different parameters.}
\label{fig:ex_level_s}
\end{figure*}

In this section we present our density model.
A natural direction for extending the normal distribution is the introduction of some skewness, and several proposals have indeed emerged, both in the univariate and multivariate case, see \cite{azzalini1985class,azzalini1996multivariate,villani2006multivariate}.
One of the most popular approaches is the Split Normal (SN) distribution, or the Split Gaussian (SG) distribution \cite{gibbons1973estimation}. In our paper we use a generalization of this model, which we call the General Split Normal (GSN) distribution. 

  We start from the one-dimensional case. After that we present a possible generalization of this definition to the multidimensional setting, which corresponds with the formula (\ref{eq:gen}). Contrary to the Split Gaussian distribution, we  skip the assumption of the orthogonality of coordinates (often called principal components), and obtain an ICA model.

\subsection{One-dimensional case}

The density of the one-dimensional Split Gaussian distribution is given by the formula
$$
SN(x;m,\sigma^2,\tau^2) = \left\{ \begin{array}{ll}
c \cdot \exp[-\frac{1}{2\sigma^2}(x-m)^2], & \textrm{for $x\leq m$},\\
c \cdot \exp[-\frac{1}{2\tau^2\sigma^2}(x-m)^2], & \textrm{for $x>m$},\\
\end{array} \right.
$$
where $c=\sqrt{\frac{2}{\pi}}\sigma^{-1}(1+\tau)^{-1}$.

As we see the split normal distribution arises from merging two opposite halves of two probability density functions of normal distributions in their common mode.
In general the use of the Split Gaussian distribution (even in 1D) allows to fit data with better precision (from the likelihood function point of view). In 1982 John \cite{john1982three} showed that the likelihood function can be expressed in an intensive form, in which the scale parameters $\sigma$ and $\tau$ are a function of the location parameter $m$ (see Theorem 3.1 proved by \cite{villani2006multivariate}). 
Thanks to this theorem we can maximize the likelihood function numerically with respect to a single parameter $m$ only. The rest of parameters are explicitly given by simple formulas.

\subsection{Multidimensional Split Gaussian distribution }
A natural generalization of the univariate split normal distribution to the multivariate settings was presented by \cite{villani2006multivariate}.
Roughly speaking, authors assume that a vector $\x \in \R^d$ follows the multivariate Split Normal distribution, if its principal components are orthogonal and follow the one-dimensional Split Normal distribution.

\begin{definition}[Definition 2.2. \cite{villani2006multivariate}]\label{def:SN}
A density of the multivariate Split Normal distribution is given by
$$
 SN_{d}(\x; \m, \Sigma,\tau)= \prod_{j=1}^{d} SN(\w_j^T(\x-\m);0,\sigma_j^2,\tau_j^2),
$$
where $\w_{j}$ is the eigenvector corresponding to the $j$-th largest eigenvalue in the spectral decomposition of $\Sigma = W \A W^{T}$ and $\m = [m_1, \ldots, m_d]^T$, $\A = \diag(\sigma_{1}^2,\ldots,\sigma_{d}^2)$ and $\tau=[\tau_{1}^2,\ldots,\tau_{d}^2]$.
\end{definition}

One can easily observe that the principal components $\w_j^T\x$ are independent.

For this generalization a similar theorem, like in the one-dimensional case, is valid. We can extract the maximum likelihood estimation by maximizing the function with respect to two parameters $\m \in \R^d$ and $W \in \M_{d}(\R)$ where columns of $W$ are orthonormal vectors ($\M_{d}(\R)$ denotes the set of $d$-dimensional square matrices). 

We may use this theorem for numerical maximization of the likelihood function w.r.t. $\m$ and $W$. Unfortunately, the optimization process on Stiefel manifold (the set of orthogonal matrices) studied by \cite{absil2009optimization}  is numerically ineffective and requires additional tools. This problem can be omitted by using Eulerian angles described by \cite{khatri1977mises}. In the two-dimensional case, $W$ is explicitly parametrized as
$$
W=
\begin{bmatrix}
\cos(\theta) & \sin(\theta) \\
-\sin(\theta) & \cos(\theta) 
\end{bmatrix}, \quad -\frac{\pi}{2} < \theta \leq \frac{\pi}{2}.
$$ 
In such a case we can straightforwardly apply standard numerical optimization algorithm.

Both of these solutions can be applied. Nevertheless, unnatural assumption of the orthogonality of principal components causes two negative effects: the optimization process is time consuming and the model with the restriction that the coordinates are orthogonal can not accommodate data as good as the general one. 
Therefore, in this article we use more flexible model -- the General Split Normal \cite{spurek2017general} distribution:

\begin{definition}\label{def:GSN}
A density of the multivariate General Split Normal distribution is given by
$$
 GSN_{d}(\x; \m,W, \sigma^2,\tau^2)=\det(W) \prod_{j=1}^{d} SN(\w_j^T(\x-\m);0,\sigma_j^2,\tau_j^2),
$$
where $\w_{j}$ is the $j$-th column of non-singular matrix $W$, $\m = (m_1, \ldots, m_d)^T$, $\sigma = (\sigma_{1},\ldots,\sigma_{d})$ and $\tau=(\tau_{1},\ldots,\tau_{d})$.
\end{definition}



Our model is a natural generalization of the multivariate Split Normal distribution proposed in \cite{villani2006multivariate} (see Definition \ref{def:SN}) and is given in the form formulated by \eqref{eq:gen}
for the set of Split Gaussian densities. 
Clearly every Split Normal distribution is a General 
Split Normal distribution.



The above generalization is flexible and allows to fit data with greater precision, see Fig. \ref{fig:rev_1}. The level sets of the GSN distribution with different parameters are presented in Fig.~\ref{fig:ex_level_s}.  
We skip the constraints of orthogonality of the principal components. Consequently, we can apply the standard optimization procedure directly. In the next section we discuss how to fit data in our model.

\begin{figure}[!t]
\normalsize
\begin{center}
\includegraphics[width=5in]{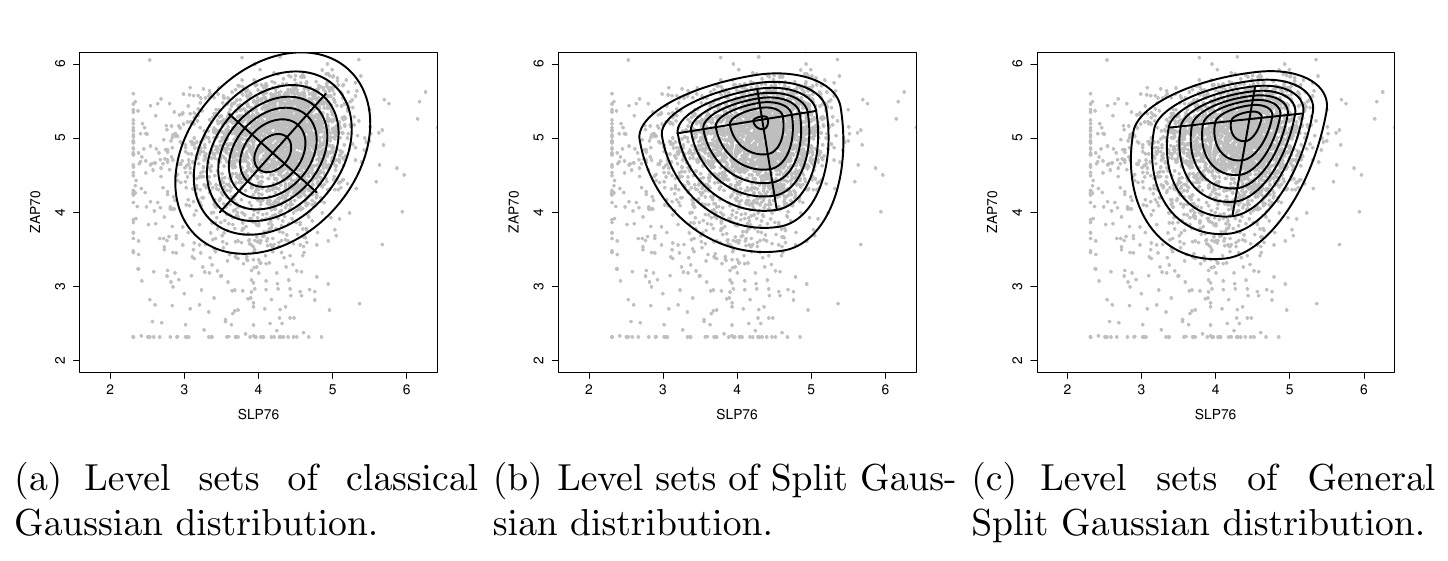}
\end{center}
\caption{Comparison between fitting Gaussian, Split Gaussian and General Split  distribution on \citep{maier2007allelic,pyne2009automated}. Observe that, contrary to Split Gaussian, General Split Gaussian does not have orthogonal basis.}
\label{fig:rev_1}
\end{figure}

\section{Maximum likelihood estimation}

In the previous section we introduced the GSN distribution. 
Now we show how to use the likelihood estimation in our setting. As it was mentioned, we have to maximize the likelihood function with respect to four parameters. In the case of the General Split Normal distribution (contrary to the classical Gaussian one) we do not have explicit formulas and consequently we heave to solve the optimization problem.

In the first subsection, we reduce our problem to the simpler one by introducing the function~${l}$. Minimization of~${l}$~is equivalent to maximization of the likelihood function.
In the second subsection we present how to minimize our function by using the gradient method.

\subsection{Optimization problem}

The density of the GSN distribution depends on four parameters $\m \in \R^d$, $W \in \M(\R^d)$, $\sigma \in \R^d$, $\tau \in \R^d$. 
We can find them by minimizing the simpler function, which depends on only  $m \in \R^d$ and $W \in \M(\R^d)$. Other parameters are given by explicit formulas.    

\begin{theorem}\label{the:min}
Let $\x_1,\ldots,\x_n$ be given.  
Then the likelihood maximized w.r.t. $\sigma$ and $\tau$ is
\begin{equation}\label{eq:1}
 \hat{L}(X;\m,W) =   \bigg( \frac{2n}{\pi e} \bigg)^{dn/2} \bigg( \frac{1}{|\det(W)|^{\frac{2}{3}}} \prod_{j=1}^{d} g_{j}(\m,W) \bigg)^{-3n/2},
\end{equation}
where
$$
\begin{array}{c}
{g}_{j}(\m,W) = {s}_{1j}^{1/3} + {s}_{2j}^{1/3},
\\[1ex]
{s}_{1j}= \! \sum\limits_{i \in I_j}[ \w_{j}^T (\x_i-\m)]^2,  {I}_j=\{ i = 1,\ldots,n \colon \w_{j}^T (\x_i-\m) \leq 0 \},
\\[1ex]
{s}_{2j}= \! \sum\limits_{i \in I_j^c}[ \w_{j}^T (\x_i-\m)]^2, {I}_j^c=\{ i = 1,\ldots,n \colon  \w_{j}^T (\x_i-\m) > 0 \},
\end{array}
$$
and the maximum likelihood estimators of $\sigma_{j}^2$ and $\tau_{j}$ are
$$\hat \sigma_j^2(\m,W) = \tfrac{1}{n} s_{1j}^{2/3} g_{j}(\m,W), \quad
\hat \tau_{j}(\m,W)=\left(\frac{s_{2j}}{s_{1j}}\right)^{1/3}.
$$
\end{theorem}

\begin{proof}
See Appendix \ref{App:A}.
\end{proof}

\begin{figure*}[t!]
\normalsize
\begin{center}
\includegraphics[width=5in]{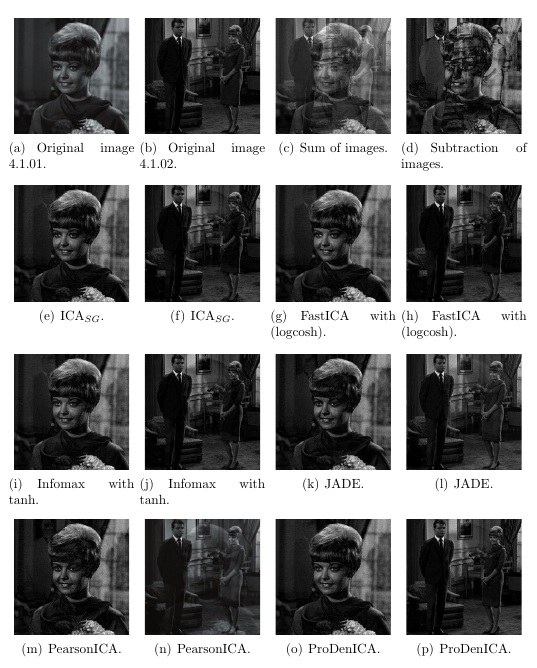}
\end{center}
\caption{Results of image separation with the uses of various ICA algorithms.}
\label{fig:image_ICA_1}
\end{figure*}

Thanks to the above theorem, instead of looking for the maximum of the likelihood function, it is enough to obtain the maximum of the simpler function~(\ref{eq:1}) which depends on two parameters $\m \in \R^d$ and $W \in \M(\R^d)$
\begin{equation}\label{equ:ll}
{l}(X;\m,W) = \frac{1}{|\det(W)|^{\frac{2}{3}}} \prod_{j=1}^{d} {g}_{j}(\m,W),
\end{equation}
where $\w_{j}$ stands for the $j$-th column of matrix $W$. 
Consequently, maximization of (\ref{eq:1}) is equivalent to minimization of  (\ref{equ:ll}), see the following corollary.

\begin{corollary}\label{c2}
Let $X \subset \R^d$, $\m \in \R^d$, $W \in \M(\R^d)$ be given, then
$$
 \argmax_{\m,W} \hat{L}(X;\m,W) =  \argmin_{\m,W} {l}(X;\m,W).
$$
\end{corollary}
%

\subsection{Gradient}



One of the possible methods of optimization is the gradient method. Since the minimum of ${l}$ is equal to the minimum of $\ln({l})$, in this subsection we calculate the gradient of $\ln({l})$. 
Before we prove suitable Theorem \ref{ther:grad}, we recall the following lemma. 

\begin{lemma}\label{jacobi}
Let $A = (a_{ij})_{1 \leq i,j \leq d}$ be a differentiable map from real numbers to $d \times d$ matrices then
\begin{equation}
\frac{\partial \det(A)}{\partial a_{ij}} = \mathrm{adj}^T(A)_{ij},
\end{equation}
where $\mathrm{adj}(A)$ stands for the adjugate of $A$, i.e. the transpose of the cofactor matrix.
\end{lemma}
\begin{proof}
By the Laplace expansion $\det A = \sum\limits_{j=1}^{d} (-1)^{i+j} a_{ij} M_{ij}$ where $M_{ij}$ is the minor of the entry in the $i$-th row and $j$-th column. Hence
$$\frac{\partial \det A}{\partial a_{ij}} = (-1)^{i+j} M_{ij} = \mathrm{adj}^T(A)_{ij}.$$
\end{proof}
Now we are ready to calculate gradient of our cost function.

\begin{theorem}\label{ther:grad}
Let $X \subset \R^d$, $\m = (\m_1, \ldots, \m_d)^T \in \R^d$, $W = (\w_{ij})_{1 \leq i,j \leq d}$ non-singular be given. 
Then
$\nabla_{\m}  \ln {l}(X;\m,W) = \left(  \frac{\partial \ln {l}(X;\m,W)}{\partial \m_1}, \ldots, \frac{\partial \ln {l}(X;\m,W)}{\partial \m_d} \right)^T$,
where
$$
\begin{array}{l}
\frac{\partial \ln {l}(X;\m,W)}{\partial \m_k} =
\sum \limits_{j=1}^d \frac{-1}{{s}_{1j}^{\frac{1}{3}} + {s}_{2j}^{\frac{1}{3}}} \bigg(
\frac{1}{3 {s}_{1j}^{\frac{2}{3}}} \sum \limits_{i \in I_j} 2 \w_j^T (\x_i - \m)  \w_{jk} + 
\frac{1}{3 {s}_{2j}^{\frac{2}{3}}} \sum \limits_{i \in I_j^c} 2 \w_j^T (\x_i - \m)  \w_{jk}
\bigg).
\end{array}
$$
Moreover,
$
\nabla_{W} \ln {l}(X;\m,W) = \left[ \frac{\partial \ln l(X;\m,W)}{\partial \w_{pk}}  \right]_{1 \leq p,k \leq d},
$
where
$$
\begin{array}{l}
\frac{\partial \ln l(X;\m,W)}{\partial \w_{pk}}  = 
-\frac{2}{3}  (\w^{-1})^T_{pk} +
\frac{1}{{s}_{1p}^{\frac{1}{3}} +{s}_{2p}^{\frac{1}{3}}} 
\bigg(
\frac{1}{3} {s}_{1p}^{-\frac{2}{3}}  \sum \limits_{i \in {I}_p} 2 \w^T_p  (\x_i - \m) (\x_{ik} - \m_k) + \\[6pt]
+ \frac{1}{3} {s}_{2p}^{-\frac{2}{3}}  \sum \limits_{i \in {I}_p^c} 2 \w^T_p  (\x_i - \m)  (\x_{ik} - \m_k) \bigg),
\end{array}
$$
and
$$
\begin{array}{c}
{s}_{1j}= \! \sum\limits_{i \in I_j}[ \w_{j}^T (\x_i-\m)]^2, {I}_j=\{ i = 1,\ldots,n \colon \w_{j}^T (\x_i-\m) \leq 0 \},
\\[1ex]
{s}_{2j}= \! \sum\limits_{i \in I_j^c}[ \w_{j}^T (\x_i-\m)]^2,  {I}_j^c=\{ i = 1,\ldots,n \colon  \w_{j}^T (\x_i-\m) > 0 \}.
\end{array}
$$
\end{theorem}

\begin{proof}
See Appendix \ref{App:B}.
\end{proof}

Thanks to the above theorem we can use gradient descent, a first-order optimization algorithm. To find a local minimum of the cost function $\ln(l)$ using gradient descent, one takes steps proportional to the negative of the gradient of the function at the current point. If instead one takes steps proportional to the positive of the gradient, one approaches a local maximum of that function, see Algorithm \ref{alg1}.

\begin{algorithm}[!h] 
\caption{:} 
\label{alg1} 
\begin{algorithmic} 
\STATE {\bf Input}
\STATE\hspace\algorithmicindent data set $X$
\STATE {\bf Initial conditions}
\STATE\hspace\algorithmicindent initialization of mean vector $\m=\mean(X)$
\STATE\hspace\algorithmicindent initialization of matrix $W = \cov(X)$ 
\STATE {\bf Gradient algorithm}
\STATE\hspace\algorithmicindent obtain new values of $\m$ and $V$ by applying gradient method for function $\log(l)$ (see formula \ref{eq:1}):
\STATE\hspace\algorithmicindent $$(\m,W) =\argmin\limits_{\bar \m,\bar W} \log({l}(X;\bar \m,\bar W)), $$ 
\STATE\hspace\algorithmicindent where 
\STATE\hspace\algorithmicindent $$\nabla_{\m}  \ln {l}(X;\m,W)$$ $$\nabla_{W} \ln {l}(X;\m,W)$$ 
\STATE\hspace\algorithmicindent are given by Theorem \ref{ther:grad}
\STATE\hspace\algorithmicindent calculate $\sigma \in \R^d$ and $\tau \in \R^d$ by using Theorem \ref{the:min}

\STATE {\bf Return value}
\STATE\hspace\algorithmicindent return optimal ICA basis $(\m,W) $.
\end{algorithmic}
\end{algorithm}

At the end of this section we present comparison of computational efficiency
between \ICA \ and various ICA methods, see Fig. \ref{fig:time_1}. In our experiment we consider the classical image separation problem, where we mixed two images by adding and subtracting them. We use ten pairs of images. Each pair was scaled to different sizes. In Fig. \ref{fig:time_1} we present mean value of computation time. FastICA, Infomax and JADE are the most effective but do not solve the problem of image separation sufficiently well, see Tab. \ref{tab:congru_img_1}. On the other hand, the ProDenICA which gives comparable result to \ICA, is much slower. 

\begin{figure*}[!t]
\normalsize
\begin{center}
\includegraphics[width=5in]{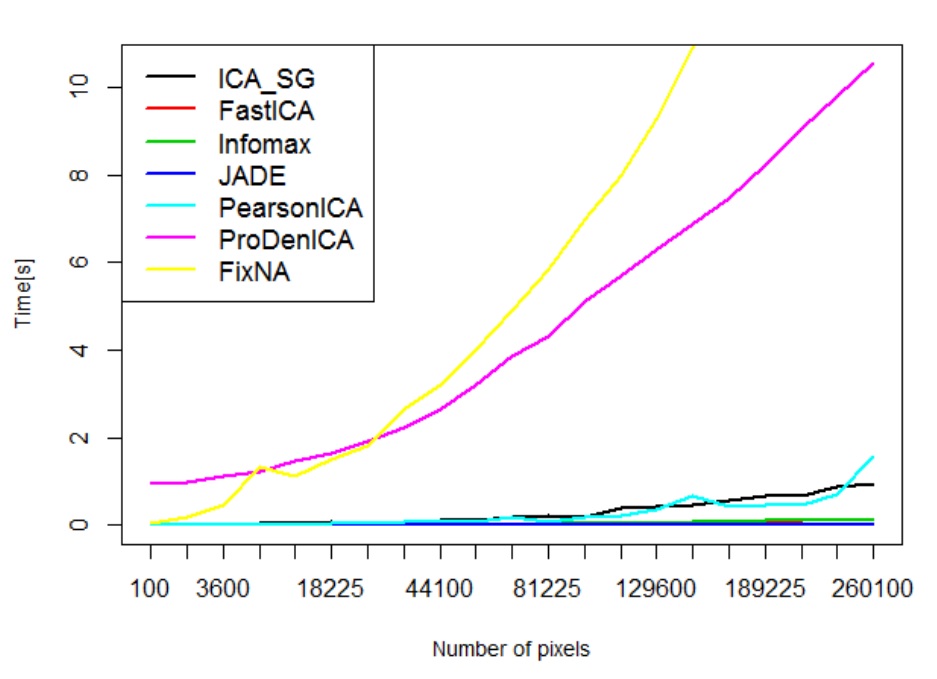}
\end{center}
\caption{Comparison of computational efficiency between \ICA \ and various ICA methods.}
\label{fig:time_1}
\end{figure*}

\section{Experiments and analysis}\label{ex}

To compare our method to classical ones we use 
Tucker's congruence coefficient \cite{lorenzo2006tucker}
(uncentered correlation) defined by
$$
Cr(\s, \bar \s) = \frac{ \sum_{i=1}^d s_i \bar \s_i}{ \sqrt{\sum_{i=1}^d \s_i^2}\sqrt{ \sum_{i=1}^d  {\bar \s}^2_i } }.
$$
Its values range between $-1$ and $+1$. It can be used to study the similarity of extracted factors across different samples. Generally, a congruence coefficient of $0.9$ indicates a high degree of factor similarity, while a coefficient of $0.95$ or higher indicates that the factors are virtually identical. 
In the case of ICA methods multiplying by the scalar any of the sources do not change results. Therefore the sign of congruence coefficient is not important and we can compare absolute value of Tucker's congruence.

%


We evaluate our method in the context of images, sound, hyperspectral unmixing and EEG data.  
For comparison we use R package {\tt ica} \cite{ica}, {\tt PearsonICA} \cite{pearsonica}, {\tt ProDenICA} \cite{prodenica}, {\tt tsBSS} \cite{tsBSS}.
The most popular method used in practice is FastICA \cite{hyvarinen1999fast,helwig2013critique} algorithm, which uses negentropy. In this context we can use three different functions to estimate neg-entropy:
logcosh, exp and kurtosis.
We also compare our method with algorithm using Information-Maximization (Infomax) approach \cite{bell1995information}. Similarly to FastICA we consider three possible nonlinear functions: hyperbolic tangent, logistic and extended Infomax.
We also consider algorithm which uses Joint Approximate Diagonalization of Eigenmatrices (JADE) proposed by Cardoso and Souloumiac's \cite{cardoso1993blind,cardoso1993blind,helwig2013critique}.

One of the most popular ICA methods dedicated for skew data is PearsonICA \cite{karvanen2000pearson,karvanen2002blind}, which minimizes mutual information using a Pearson \cite{stuart1968advanced} system-based parametric model. Another model we consider is ProDenICA \cite{bach2002kernel,hastie2009elements}, which is based not on a
single nonlinear function, but on an entire function space of candidate nonlinearities. In particular, the method works with the functions in a reproducing kernel Hilbert space, and make use of the “kernel trick” to search over this space efficiently. 
We also compare our method with  FixNA \cite{shi2009blind}, method for blind source separation problem.

%
\subsection{Separation of images}

One of the most popular application of ICA is the separation of images. In our experiments we use four images from the USC-SIPI Image Database of size $256 \times 256$ pixels (4.1.01, 4.1.06, 4.1.02, 4.1.03) and eight of size $512 \times 512$ pixels (4.2.04, 4.2.02, boat.512, elaine.512, 5.2.10, 5.2.08, 5.3.01, 4.2.03). We also use 8 images from the Berkeley Segmentation Dataset of size $482 \times 321$ with indexes (\#119082, \#42049, \#43074, \#38092, \#157055, \#220075, \#295087, \#167062). We make random pairs of above images and use them as a source signal, combined by the mixing matrix $A = \begin{bmatrix} 1 & 1  \\ 1 & -1  \end{bmatrix} $. From practical point of view, we simply obtain two new images by adding and dividing sources pictures. Our goal is to reconstruct original images by using only the knowledge about mixed ones. The visualization of this process we present in Fig. \ref{fig:image_ICA_1}. The results of this experiment are presented in Tab.~\ref{tab:congru_img_1} where we exhibit Tucker's congruence coefficients.

In the case of the Tucker's congruence coefficient measure almost in all situation we obtain better results. The \ICA \ method essentially better recovers original signals. In Fig.~\ref{fig:image_ICA_1}  we can sow that \ICA \ almost perfectly recovers source signal.


\begin{table*}[!t]
\centering
\scalebox{0.575}{ 
\begin{tabular}{ | c | c  | c c c | c c c | c | c | c | c | }
\multicolumn{1}{c}{} & \multicolumn{1}{c}{\ICA}  & \multicolumn{3}{c}{FastICA} & \multicolumn{3}{c}{Infomax} & \multicolumn{1}{c}{JADE} & \multicolumn{1}{c}{PearsonICA} & \multicolumn{1}{c}{ProDenICA} & \multicolumn{1}{c}{FixNA} \\ 
 &  &  logcosh & exp & kurtosis & tanh & tangent & logistic &  & & &  \\
\hline
4.1.01 & \bf-0.9818  &  0.5481  &  -0.5457  &  -0.5485  &  0.548  &  -0.5484  &  -0.548  &  -0.5492  &  -0.5308  &  -0.0013  &  0.5503  \\
4.1.02 & \bf0.992  &  0.6696  &  0.6644  &  0.6707  &  0.6695  &  0.6705  &  0.6695  &  0.6726  &  0.6696  &  -0.0981  &  -0.6761  \\ \hline
4.1.06 & \bf-0.9609  &  -0.4297  &  -0.4297  &  -0.4296  &  -0.4297  &  -0.4297  &  -0.4296  &  -0.4296  &  -0.4297  &  0.4297  &  0.0148  \\
4.1.03 & \bf0.5664  &  0.2062  &  0.2062  &  0.2057  &  0.2061  &  0.206  &  0.2058  &  0.2058  &  -0.2062  &  0.207  &  0.0127  \\ \hline
4.2.04 & \bf-0.5034  &  0.0506  &  0.0528  &  -0.0499  &  0.0505  &  -0.0512  &  0.0508  &  0.0397  &  0.3123  &  -0.3164  &  0.1461  \\
5.2.10 & 0.2893  &  -0.0719  &  -0.0749  &  0.0709  &  -0.0717  &  0.0727  &  -0.0722  &  -0.057  &  -0.4275  &  \bf0.4334  &  -0.1979  \\ \hline
4.2.02 & \bf0.2305  &  -0.0376  &  0.0203  &  -0.0017  &  0.0377  &  0.0265  &  0.0061  &  -0.0093  &  -0.1228  &  0.1282  &  0.1235  \\
5.2.08 & \bf0.5717  &  0.1037  &  -0.0625  &  -0.0097  &  -0.1039  &  -0.0773  &  -0.0285  &  0.0086  &  -0.2913  &  -0.3091  &  -0.2931  \\ \hline
boat.512 & \bf 0.3593  &  0.0351  &  0.0314  &  -0.056  &  0.0343  &  -0.0449  &  0.0298  &  0.0356  &  -0.1046  &  -0.0461  &  0.3175  \\
5.3.01 & 0.4316  &  0.0078  &  0.0138  &  -0.0262  &  0.0091  &  -0.008  &  0.0164  &  0.007  &  0.1061  &  0.0486  &  \bf -0.5303  \\ \hline
elaine.512 & \bf0.5874  &  0.32  &  0.32  &  -0.32  &  0.32  &  -0.32  &  0.32  &  0.32  &  -0.32  &  0.0287  &  0.2282  \\
4.2.03 & -0.0226  &  -0.3196  &  -0.3196  &  0.3201  &  -0.3196  &  0.3199  &  -0.3196  & \bf -0.3202  &  -0.3195  &  -0.048  &  -0.2554  \\ \hline
119082 & \bf0.9987  &  0.5736  &  0.5736  &  0.5731  &  0.5737  &  0.5733  &  0.5735  &  0.5735  &  -0.032  &  0.5744  &  0.3695  \\
157055 & \bf0.389  &  -0.3619  &  -0.3619  &  -0.3618  &  -0.3619  &  -0.3619  &  -0.3619  &  -0.3619  &  0.0046  &  0.3619  &  -0.2446  \\ \hline
42049 & \bf -0.7493  &  0.3009  &  0.3028  &  -0.299  &  -0.3005  &  -0.3031  &  -0.3007  &  -0.2898  &  0.2596  &  0.0421  &  0.142  \\
220075 & 0.4359  &  -0.5087  &  -0.5154  &  0.503  &  0.5074  &  \bf0.5168  &  0.5081  &  0.4789  &  0.4838  &  -0.0645  &  -0.1839  \\ \hline
43074 & \bf-0.7371  &  0.0344  &  0.0323  &  0.0429  &  0.0348  &  0.0404  &  0.0342  &  0.0324  &  0.0891  &  0.3925  &  0.2458  \\
295087 & -0.3997  &  -0.048  &  -0.0458  &  -0.0566  &  -0.0484  &  -0.0541  &  -0.0478  &  -0.0459  &  -0.1035  & \bf 0.4015  &  -0.2406  \\ \hline
38092 & \bf -0.5949  &  0.0555  &  0.0564  &  0.031  &  -0.0553  &  0.041  &  0.0557  &  0.0375  &  0.0535  &  0.4036  &  0.2614  \\
167062 & 0.3255  &  -0.0025  &  -0.0041  &  0.0425  &  0.0021  &  0.0241  &  -0.0029  &  0.0306  &  0.0011  &  \bf 0.7404  &  -0.5495  \\ \hline

\end{tabular}
}
\caption{The Tucker's congruence coefficient measure between original images and results of different ICA algorithms.}
\label{tab:congru_img_1}
\end{table*}

\subsection{Cocktail-party problem}
In this subsection we compare our method with classical ones in the case of cocktail-party problem. 
Imagine that you are in a room where two people are speaking simultaneously. You have two microphones, which you hold in different locations. The microphones give you two recorded time signals, which we could interpret as mixed signal $\x$. Each of these recorded signals is a weighted sum of the speech signals emitted by the two speakers, which we denote by $\s$. 
The cocktail-party problem is to estimate the two original speech signals. 

In our experiments we use signal obtained by mixing synthetic sources\footnote{We use signals from \url{http://research.ics.aalto.fi/ica/cocktail/cocktail_en.cgi}.} (similar as before we use mixing matrix $A = \begin{bmatrix} 1 & 1  \\ 1 & -1  \end{bmatrix} $). 
Comparison between methods we present in Tab. \ref{tab:congru_sound_1}. In the case of cocktail-party problem our method recovers  
sources signal better then classical methods. 


\begin{table*}[!t]
\centering
\scalebox{0.6}{ 
\begin{tabular}{ | c | c  | c c c | c c c | c | c | c | c | }
\multicolumn{1}{c}{} & \multicolumn{1}{c}{\ICA}  & \multicolumn{3}{c}{FastICA} & \multicolumn{3}{c}{Infomax} & \multicolumn{1}{c}{JADE} & \multicolumn{1}{c}{PearsonICA} & \multicolumn{1}{c}{ProDenICA} & \multicolumn{1}{c}{FixNA} \\ 
 &  &  logcosh & exp & kurtosis & tanh & tangent & logistic & & & & \\
\hline
source 1 & \bf 0.1597  &  0.1097  &  0.1096  &  0.1101  &  0.1097  &  0.11  &  0.1097  &  0.1101  &  0.1097  &  0.1412  &  0.109  \\
source 2 & 0.7739  &  0.7705  &  0.7713  &  0.7672  &  0.7705  &  0.7685  &  0.7705  &  0.7704  &  0.7704  &  \bf 0.9998  &  0.7751  \\ \hline
source 2 & \bf0.1388  &  0.0899  &  0.0899  &  0.0908  &  0.0899  &  0.0899  &  0.0899  &  0.0908  &  0.0899  &  0.0984  &  0.0907  \\
source 3 & 0.9435  &  0.9075  &  0.9076  &  0.898  &  0.9074  &  0.907  &  0.9074  &  0.9075  &  0.9075  &  \bf0.9989  &  0.8988  \\ \hline
source 3 & \bf0.1985  &  0.079  &  0.0791  &  0.079  &  0.079  &  0.079  &  0.079  &  0.079  &  0.0789  &  0.0843  &  0.0791  \\
source 4 & 0.8453  &  0.8887  &  0.8882  &  0.8889  &  0.8887  & \bf0.8892  &  0.8887  &  0.8898  &  0.8898  &  0.8459  &  0.8882  \\ \hline
source 4 & \bf0.232  &  0.0989  &  0.0989  &  0.099  &  0.0989  &  0.0989  &  0.0989  &  0.099  &  0.0989  &  0.1153  &  0.0989  \\
source 5 & 0.7679  &  0.7798  &  0.7799  &  0.7793  &  0.7798  &  0.7798  &  0.7798  &  0.7801  &  0.7801  &  \bf0.9344  &  0.7796  \\ \hline
source 5 & \bf0.1728  &  0.0989  &  0.099  &  0.0988  &  0.0989  &  0.0989  &  0.0989  &  0.0989  &  0.0989  &  0.0963  &  0.0987  \\
source 6 & 0.9424  &  0.9245  &  0.9243  &  0.9256  &  0.9246  &  0.925  &  0.9246  &  0.9245  &  0.9245  &  \bf0.9729  &  0.9273  \\ \hline
source 6 & \bf0.15  &  0.0404  &  0.0404  &  0.0402  &  0.0404  &  0.0404  &  0.0404  &  0.0402  &  0.0404  &  0.0567  &  0.0402  \\
source 7 & 0.7417  &  0.7129  &  0.7134  &  0.707  &  0.7132  &  0.7125  &  0.7129  &  0.7124  &  0.7124  &  \bf0.9998  &  0.7099  \\ \hline
source 7 & \bf0.1036  &  0.0839  &  0.084  &  0.0839  &  0.0839  &  0.0839  &  0.0839  &  0.0839  &  0.084  &  0.093  &  0.0836  \\
source 8 & 0.908  &  0.9016  &  0.9015  &  0.9019  &  0.9019  &  0.9019  &  0.9017  &  0.9014  &  0.9014  &  \bf0.9999  &  0.9056  \\ \hline
source 8 & 0.1166  &  0.1153  &  0.1156  &  0.1145  &  0.1152  &  0.1148  &  0.1153  &  0.1155  &  0.1149  &  \bf0.1427  &  0.1147  \\
source 9 & 0.8212  &  0.8136  &  0.8116  &  0.8195  &  0.8141  &  0.8174  &  0.8138  &  0.8165  &  0.8165  &  \bf0.9996  &  0.8176  \\ \hline

\end{tabular}
}
\caption{The Tucker's congruence coefficient measure between original sound and results of different ICA algorithms in the case of cocktail-party problem.}
\label{tab:congru_sound_1}
\end{table*}

\subsection{Hyperspectral Unmixing}

Independent component analysis has been recently
applied into hyperspectral unmixing as a result of its low
computation time and its ability to perform without prior information.
However, when applying ICA for hyperspectral unmixing,
the independence assumption in the ICA model conflicts with
the abundance sum-to-one constraint and the abundance nonnegative
constraint in the linear mixture model, which affects the
hyperspectral unmixing accuracy. Nevertheless, ICA was recently applied in this area \cite{wang2015abundance,caiafa2008blind}. In this subsection we apply simple example which shows that our method can by used for spectral data.

Urban data  \cite{fyzhu2014IJPRSSSNMF,fyzhu2014TIPDgSNMF,fyzhu2014JSTSPRRLbSF} is one of the most widely used hyperspectral data-sets used in the hyperspectral unmixing study. Each image has $307 \times 307$ pixels, each of which corresponds to a $2 \times 2$ m area. In this image, there are 210 wavelengths ranging from 400 nm  to 2500 nm, resulting in a spectral resolution of 10 nm. After the channels 1--4, 76, 87, 101--111, 136--153 and 198--210 are removed (due to dense water vapor and atmospheric effects), there remain 162  channels (this is a common preprocess for hyperspectral unmixing analyses). There is ground truth \cite{fyzhu2014IJPRSSSNMF,fyzhu2014TIPDgSNMF,fyzhu2014JSTSPRRLbSF}, which contains 4 channels: \#1 Asphalt, \#2 Grass, \#3 Tree and \#4 Roof.

A highly mixed area is cut from the original data set in this experiment (similar example was showed in \cite{wang2015abundance}), with the size of $200 \times 150$ pixels. 

In our experiment we apply various ICA methods and report the Tucker's congruence coefficient measure between each layer and the closest reference channel, see Fig. \ref{fig:spec_1}. \ICA \ and ProDenICA give layers which contain more information then the other approaches. Distance between four best  channels to the reference ones we present in Tab. \ref{tab:spec}. 

\begin{table*}[!t]
\centering
\scalebox{0.7}{ 
\begin{tabular}{ | c | c  | c c c | c c c | c | c | }
\multicolumn{1}{c}{} & \multicolumn{1}{c}{\ICA}  & \multicolumn{3}{c}{FastICA} & \multicolumn{3}{c}{Infomax}  & \multicolumn{1}{c}{PearsonICA} & \multicolumn{1}{c}{ProDenICA}  \\ 
 &  &  logcosh & exp & kurtosis & tanh & tangent & logistic & &  \\
\hline
\#1 Asphalt  &\bf 0.6774 &  0.2859 & 0.2864 & -0.2595 & -0.2972 & -0.2954 &  -0.2972 &  0.20978 & 0.4928  \\
\#2 Grass  & \bf -0.7784 &  -0.2746 & -0.2605 &  -0.2798 &  -0.2814 & -0.2816 & -0.2814 &  -0.2412 & -0.4323  \\
\#3 Tree  & \bf 0.7267 & 0.2338 &  0.2717 &  -0.2547 &  0.2441 & 0.2354 &  0.2442 &   0.2482 & -0.5961  \\
\#4 Roof &\bf 0.6666 &  -0.4256 &  0.4279 &  0.4167 & -0.4244 &  0.4301 & -0.4244 &  0.4193 &  -0.6128  \\

\end{tabular}
}
\caption{The Tucker's congruence coefficient measure between reference layers and results of different ICA algorithms in the case of the urban data set.}
\label{tab:spec}
\end{table*}

\begin{figure*}[!t]
\normalsize
\begin{center}
\includegraphics[width=5in]{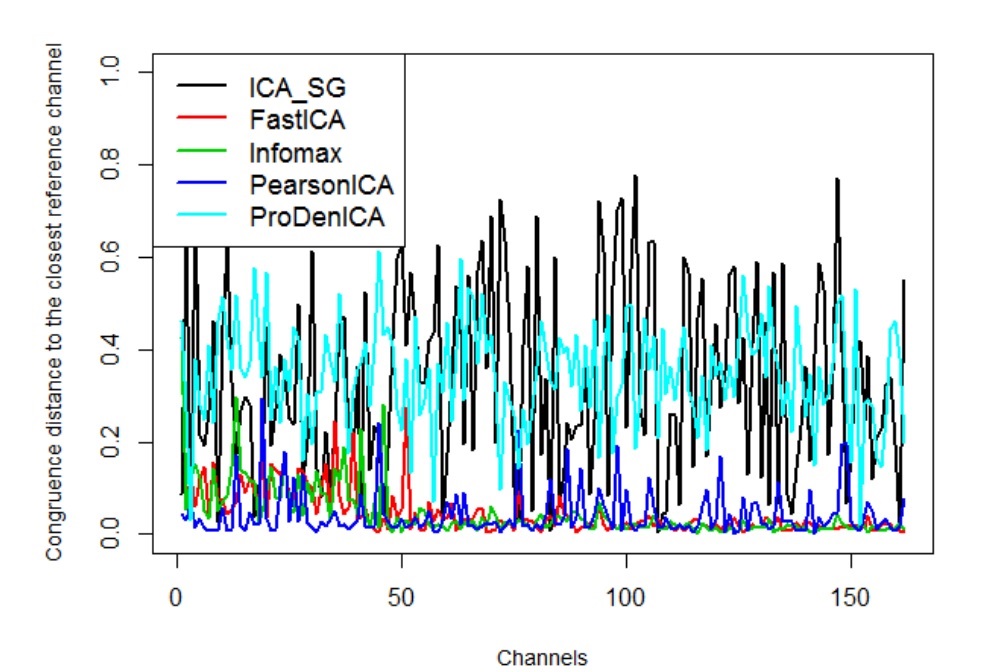}
\end{center}
\caption{Congruence distance between layers obtain by different ICA algorithms and the closest reference channel.}
\label{fig:spec_1}
\end{figure*}


\subsection{EEG}

At the end of this section we present how our method works in the case of EEG signals. In this context, ICA is applied to many different task like
eye movements, blinks, muscle, heart and line noise e.t.c.. 
In this experiment we concentrate on eye movement and blink artifacts. 
Our goal here is to demonstrate that our method is capable of
finding artifacts in real EEG data. However,
we emphasize that it does not provide a complete solution
to any of these practical problems. Such a solution usually
entails a significant amount of domain-specific knowledge
and engineering. Nevertheless, from these preliminary
results with EEG data, we believe that
the method presented in this paper provides a reasonable
solution for signal separation, which is simple and
effective enough to be easily customized for a broad range
of practical problems.

For EEG analysis, the rows of the input matrix $\x$ are the EEG signals recorded at
different electrodes, the rows of the output data matrix $\s = W\x$ are time courses of
activation of the ICA components, and the columns of the inverse matrix, $W$, give
the projection strengths of the respective components onto the scalp sensors. 

One EEG data set used in the analysis was collected from 40 scalp electrodes (see Fig. \ref{fig:EEG} a)). The second and the third are located very near to eye and can be understood as a base (we can use them for removing eye blinking artifacts). In Fig. \ref{fig:EEG} b) we present signals obtained by \ICA.  The scale of this figure is large but we can find the data which have spikes exactly in the same place as the two base signals (see Fig. \ref{fig:EEG} c)). After removing selected signal and going back to the original situation we obtain signal (see Fig. \ref{fig:EEG} d)) without eye blinking artifacts (compare Fig. \ref{fig:EEG} a) with  Fig. \ref{fig:EEG} d)).

\begin{figure*}[!t]
\normalsize
\begin{center}
\includegraphics[width=5in]{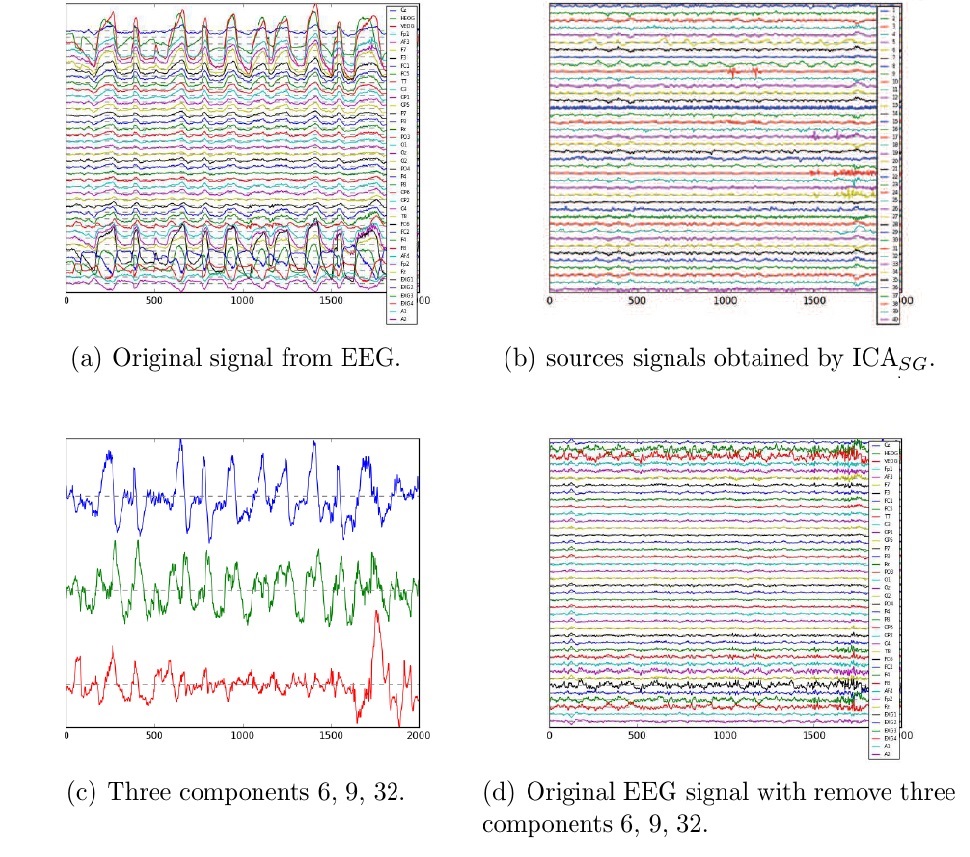}
\end{center}
\caption{Results of \ICA in the case of EEG data.}
\label{fig:EEG}
\end{figure*}


\section{Conclusion}

In our work we introduce and explore a new approach to ICA which is based on the asymmetry of the data. 
Roughly speaking in our approach, instead of approximating the data by product of densities with heavy tails, we approximate it by a product of 
asymmetric densities -- the Split Gaussian distribution.
Contrary to classical approaches which consider third or fourth central moment, our algorithm in practice is based on second moments. This is a consequence of the fact that Split Gaussian distributions arise from merging two opposite halves of normal distributions in their common mode. Therefore  we use only second order moments to describe skewness in dataset, and therefore we obtain an effective ICA method which is resistant to outliers.

We verified our approach on images, sound and EEG data.
In the case of source signal reconstructing our approach gives essentially better results (better recover original signals).
The main reason is such that kurtosis is very sensitive to the outliers and that the asymmetry of the data is more popular than heavy tails in real data sets.




\section{Appendix A}\label{App:A}

\begin{proof}[Proof of Theorem \ref{the:min}.]
Let $X=\{ \x_1, \ldots, \x_n \}$.
We write 
$$
\z_i= W(\x_i-m), \quad \z_{ij}= \w_j^T(\x_i-m),
$$
for observation $i$, where $i=1,\ldots,n$ and coordinates $j=1,\ldots,d$.

Let us consider the likelihood function, i.e. 
$$
\begin{array}{l}
L(X;\m,W,\sigma,\tau) = \prod\limits_{i=1}^{n} GSN_d(\x_i ; \m,W,\sigma,\tau) \\[6pt] 
=\prod\limits_{i=1}^{n} | \det(W)|  \prod\limits_{j=1}^{d} SN(  \w_j^T (\x_i - \m) ; 0 , \sigma_j^2, \tau_j^2)\\[6pt]
=\Big( c_1|\det(W)| \Big)^{n} \Big( \prod\limits_{j=1}^{d} \sigma_j(1+\tau_j) \Big)^{-n} 
\prod\limits_{i=1}^{n} \prod\limits_{j=1}^{d} \exp \Big[ -\frac{1}{2\sigma_j^2}z_{ij}^2 (\1_{ \{ z_{ij} \leq 0 \} } + \tau_{j}^{-2} \1_{ \{ z_{ij} > 0 \} }) \Big],
\end{array}
$$
where 
$
c_1=\left( \sqrt{\tfrac{2}{\pi}} \right)^{d}.
$
Now we take the log-likelihood function, i.e.
$$
\begin{array}{l}
\ln(L(X;\m,W,\sigma,\tau)) \\[6pt]
=\ln \bigg( \Big( c_1|\det(W)| \Big)^{n} \Big( \prod\limits_{j=1}^{d} \sigma_j(1+\tau_j) \Big)^{-n} \bigg) + 
 \sum\limits_{i=1}^{n} \sum\limits_{j=1}^{d} \Big[ -\frac{1}{2\sigma_j^2}z_{ij}^2 (\1_{ \{ z_{ij} \leq 0 \} } + \tau_{j}^{-2} \1_{ \{ z_{ij} > 0 \} })\Big]  \\[6pt]
= \ln \bigg( \Big( c_1|\det(W)| \Big)^{n} \Big( \prod\limits_{j=1}^{d} \sigma_j(1+\tau_j) \Big)^{-n} \bigg)  -
  \frac{1}{2} \sum\limits_{j=1}^{d} \Big( \sigma_j^{-2} \sum\limits_{i \in I_{j}}    z_{ij}^2   + \frac{\sigma_j^{-2}}{\tau_{j}^{2} }  \sum\limits_{i \in I_{j}^{c}}   z_{ij}^2  \Big) \\[6pt]
= \ln \bigg( \Big( c_1|\det(W)| \Big)^{n} \Big( \prod\limits_{j=1}^{d} \sigma_j(1+\tau_j) \Big)^{-n} \bigg)  - 
 \sum\limits_{j=1}^{d} \frac{1}{2\sigma_j^{2}} \Big(  s_{1j}  + \frac{1}{\tau_{j}^{2} }  s_{2j}  \Big).
\end{array}
$$

We fix  $\m$, $W$ and maximize the log-likelihood function over $\tau$ and $\sigma$.
In such a case we have to solve the following system of equations
$$
\begin{array}{l}
\frac{\partial  \ln ( L(X;\m,W,\sigma,\tau) ) }{\partial \sigma_j} = -\frac{n}{\sigma_j} +  \sigma_j^{-3} (s_{1j} + \tau_j^{-2} s_{2j} )
 =0, \\[6pt] 
 \frac{\partial  \ln ( L(X;\m,W,\sigma,\tau) ) }{\partial \tau_j} = - \frac{n}{1+\tau_j} + \frac{s_{2j}}{\tau_j^{3}\sigma_j^{2}} =0 , 
\end{array}
$$
for  $ j=1,\ldots,d$.
By simple calculations we obtain the expressions for the estimators
\begin{align*}
\hat{\sigma}_j^2(\m,W) = 
\tfrac{1}{n} s_{1j}^{2/3} g_{j}(\m,W), \qquad
\hat{\tau}_{j}(\m,W) = \bigg( \frac{s_{2j}}{s_{1j}} \bigg)^{1/3}.
\end{align*}
Substituting it into the log-likelihood function,
we get
$$
\begin{array}{l}
\hat{L}(\m,W) = \bigg( \frac{2}{\pi} \bigg)^{\frac{dn}{2}}  |\det(W)|^{n} \cdot \Big( \prod\limits_{j=1}^{d} \frac{1}{\sqrt{n}} g_j(\m,W)^{\frac{3}{2}} \Big)^{-n}  e^{-\frac{dn}{2}}\\[6pt]
= \bigg( \frac{2n}{\pi e} \bigg)^{\frac{dn}{2}}  \Big( \frac{1}{|\det(W)|^{\frac{2}{3}}} \prod\limits_{j=1}^{d} g_j(\m,W) \Big)^{-\frac{3n}{2}}. 
\end{array}
$$
\end{proof}

\section{Appendix B}\label{App:B}

\begin{proof}[Proof of Theorem \ref{ther:grad}.]
Let us start with the partial derivative of $\ln({l})$ with respect to $\m$. We have
$$
\begin{array}{l}
\frac{\partial \ln {l}(X;\m,W)}{\partial \m_k} =
\sum \limits_{j=1}^d \frac{\partial \ln ({g}_j(\m,W))}{\partial \m_k} = \sum\limits_{j=1}^d \frac{1}{{s}_{1j}^{\frac{1}{3}} + {s}_{2j}^{\frac{1}{3}}} \frac{\partial ({s}_{1j}^{\frac{1}{3}} + {s}_{2j}^{\frac{1}{3}})}{\partial \m_k} 
 \sum \limits_{j=1}^d \frac{1}{{s}_{1j}^{\frac{1}{3}} + {s}_{2j}^{\frac{1}{3}}} \bigg(
\frac{1}{3 {s}_{1j}^{\frac{2}{3}}} \frac{\partial {s}_{1j}}{\partial \m_k} +
\frac{1}{3 {s}_{2j}^{\frac{2}{3}}} \frac{\partial {s}_{2j}}{\partial \m_k}
\bigg).
\end{array}
$$
Now, we need $\frac{\partial {s}_{1j}}{\partial \m_k}$ and $\frac{\partial {s}_{2j}}{\partial \m_k}$, therefore
$$
\begin{array}{l}
\frac{\partial {s}_{1j}}{\partial \m_k} = 
\sum\limits_{i \in {I}_j} \frac{\partial [\w^T_j (\x_i - \m)]^2}{\partial \m_k} = \sum\limits_{i \in {I}_j} 2 \w^T_j (\x_i - \m) \frac{\partial \w^T_j (\x_i - \m)}{\partial \m_k} = 
 \sum\limits_{i \in {I}_j} - 2 \w^T_j (\x_i - \m) \w_{jk}.
\end{array}
$$
Analogously we get
$$
\begin{array}{l}
\frac{\partial {s}_{2j}}{\partial \m_k} = \sum\limits_{i \in {I}_j^c} -2 \w^T_j (\x_i - \m) \w_{jk}.
\end{array}
$$
Hence 
$$
\begin{array}{l}
\frac{\partial \ln {l}}{\partial \m_k} =\sum\limits_{j=1}^d \frac{-1}{{s}_{1j}^{\frac{1}{3}} + {s}_{2j}^{\frac{1}{3}}} \bigg(
\frac{1}{3 {s}_{1j}^{\frac{2}{3}}} \sum\limits_{i \in I_j} 2 \w_j^T (\x_i - \m)  \w_{jk} +
\frac{1}{3 {s}_{2j}^{\frac{2}{3}}} \sum\limits_{i \in I_j^c} 2 \w_j^T (\x_i - \m) \w_{jk}
\bigg).
\end{array}
$$

Now we calculate the partial derivative of $\ln {l}(X;\m,W)$ with respect to the matrix $W$. We have
$$
\begin{array}{l}
\frac{\partial \ln {l}(X;\m,W)}{\partial \w_{pk}} = \frac{\partial \ln |\det(W)|^{-\frac{2}{3}}}{\partial \w_{pk}} + \sum\limits_{j=1}^d \frac{\partial \ln ({g}_j(\m,W))}{\partial \w_{pk}}.
\end{array}
$$
To calculate the derivative of the determinant we use Jacobi's formula (see Lemma \ref{jacobi}).
Hence
$$
\begin{array}{l}
\frac{\partial \ln (\det(W)^{-\frac{2}{3}})}{\partial \w_{pk}} = \det(W)^{\frac{2}{3}}  \Big(-\frac{2}{3}\Big)  \det(W)^{-\frac{5}{3}}  \frac{\partial \det(W)}{\partial \w_{pk}} = -\frac{2}{3} \det(W)^{-1}  \mathrm{adj}^T(W)_{pk} \\[6pt]
 = -\frac{2}{3} \frac{1}{\det(W)}  \left[\det(W)  (W^{-1})^T_{pk}\right]= -\frac{2}{3}  (\w^{-1})^T_{pk},
\end{array}
$$
where $(\w^{-1})^T_{pk}$ is the element in the $p$-th row and $k$-th column of the matrix $(W^{-1})^T$. Now we calculate 
$$
\begin{array}{l}
\frac{\partial \ln ({g}_j(\m,W))}{\partial \w_{pk}} = \frac{1}{{s}_{1j}^{\frac{1}{3}} + {s}_{2j}^{\frac{1}{3}}} \frac{\partial ({s}_{1j}^{\frac{1}{3}} + {s}_{2j}^{\frac{1}{3}})}{\partial \w_{pk}}= \frac{1}{{s}_{1j}^{\frac{1}{3}} + {s}_{2j}^{\frac{1}{3}}} \bigg(
\frac{1}{3 {s}_{1j}^{\frac{2}{3}}}  \frac{\partial {s}_{1j}}{\partial \w_{pk}} +
\frac{1}{3 {s}_{2j}^{\frac{2}{3}}}  \frac{\partial {s}_{2j}}{\partial \w_{pk}}
\bigg),
\end{array}
$$
where
$$
\begin{array}{l}
\frac{\partial {s}_{1j}}{\partial \w_{pk}} = \sum\limits_{ i \in {I}_j} \frac{\partial [\w^T_j (\x_i - \m)]^2}{\partial \w_{pk}} = \sum\limits_{ i \in {I}_j} 2 \w^T_j (\x_i - \m) \frac{\partial \w^T_j (\x_i - \m)}{\partial \w_{pk}}=
\\[6pt]
\left\{ \begin{array}{ll}
0, & \text{if} \; j\neq p\\
\sum\limits_{ i \in {I}_p} 2 \w^T_p (\x_i - \m) (\x_{ik} - \m_k), & \text{if} \; j=p\\
\end{array} \right.
\end{array}
$$
and $\x_{ik}$ is the $k$-th element of the vector $\x_i$. Analogously we get
$$\frac{\partial {s}_{2j}}{\partial \w_{pk}} = \left\{ \begin{array}{ll}
0, & \text{if} \; j\neq p,\\
\sum\limits_{ i \in {I}_p^c} 2 \w^T_p (\x_i - \m)  (\x_{ik} - \m_k), & \text{if} \; j=p.
\end{array} \right.
$$
Hence we obtain 
$$
\begin{array}{l}
\frac{\partial \ln {l}}{\partial \w_{pk}} = -\frac{2}{3} (\w^{-1})^T_{pk} + \frac{1}{{s}_{1p}^{\frac{1}{3}} +{s}_{2p}^{\frac{1}{3}}} 
 \bigg(
\frac{1}{3} {s}_{1p}^{-\frac{2}{3}} \sum\limits_{ i \in {I}_p} 2 \w^T_p (\x_i - \m) (\x_{ik} - \m_k)\\[6pt]
+ \frac{1}{3} {s}_{2p}^{-\frac{2}{3}} \sum\limits_{ i \in {I}_p^c} 2 \w^T_p (\x_i - \m)   (\x_{ik} - \m_k) \bigg).
\end{array}
$$

\end{proof}







\section*{Acknowledgment}

Research of P. Spurek was supported by the National Center of Science
(Poland) grant no. 2015/19/D/ST6/01472. 
Research of J. Tabor was supported by the National Center of Science
(Poland) grant no. UMO-2014/13/B/ST6/01792.



\end{document}